\documentclass[a4paper,12pt]{article}
\usepackage{latexsym}
\usepackage{amsmath}
\usepackage{amssymb}
\usepackage{enumerate}

\usepackage{theorem}
\newtheorem{theorem}{Theorem}[section]
\newtheorem{proposition}[theorem]{Proposition}
\newtheorem{lemma}[theorem]{Lemma}

\theorembodyfont{\rmfamily}
\newtheorem{proof}{\textmd{\textit{Proof.}}}

\newtheorem{remark}[theorem]{Remark}

\makeatletter

\@addtoreset{equation}{section}
\makeatother

\newcommand{\qedd}{\hfill \Box}

\title{Necessary and sufficient conditions for two dimensional $(\alpha,\beta)$-metrics with reversible geodesics.\footnote{
Mathematics Subject Classification (2010)\,:\,53B40, 53C22.}
\footnote{
Keywords: geodesics, moving frames, isothermal coordinates}
}
\author{Ioana M. Masca, Sorin V. Sabau and Hideo Shimada}
\date{}
\begin{document}


\maketitle

\begin{abstract}
We study the necessary and sufficient conditions for a Finsler surface with $(\alpha,\beta)$-metrics to be with reversible geodesics.
\end{abstract}

\section{Introduction}

The study of Finsler metrics with reversible geodesics is an interesting topic in Finsler geometry. 
Due their computational advantages, we consider this problem only for Finsler spaces with $(\alpha, \beta)$-metrics, the more general cases remaining to be studied in future.

The Randers case was studied by M. Crampin in \cite{Cr} and the more general case of $(\alpha, \beta)$ - metrics by the authors in \cite{MSS}. However, the approach used it is not
suitable for the study of the 2-dimensional case.

In the present paper we give the necessary and sufficient conditions for a 2-dimensional Finsler space with $(\alpha, \beta)$-metrics to be with reversible geodesics using a completely different
approach than the one used in \cite{MSS}. 


\section{Finsler surfaces}

Let us recall that a Finsler surface is a pair $(M, F)$, where $M$ is a real smooth 2-dimensional manifold and $F: TM \longrightarrow \lbrack 0, \infty)$ a Finsler norm, i.e. a 
positive, smooth function on $\widetilde{TM} = TM\setminus \{0\}$, with the homogeneity proprety $F(x, \lambda y) = \lambda\cdot F(x, y)$, for all $\lambda > 0$ and all 
$(x, y)\in \widetilde{TM}$ and whose Hessian matrix 
\begin{eqnarray*}
g_{ij} = \frac{1}{2}\frac{\partial^2F^2}{\partial y^i\partial y^j}
\end{eqnarray*}
is positive definite at each point $u = (x, y) \in \widetilde{TM}$.

Equivalently, a Finsler structure on the surface $M$ can be regarded as a smooth hypersurface $\Sigma^3 \subset TM$ for which the canonical projection 
$\pi: \Sigma \longrightarrow TM$ is a surjective submersion having the property that for each $x \in M$, the $\pi$-fiber $\Sigma_x = \pi^{-1}(x)$ is a smooth, closed, strongly
convex curve in $T_xM$ enclosing the origin.

Here, strongly convex means that $\Sigma_x$ is strictly convex and it has contact of precisely order two with its tangent line in each point.
Traditionally, the curve $\Sigma_x \subset TM$ is called {\it the indicatrix} of the Finsler structure $F$ and it has the property that it is not centrally symmetric about the
origin of $T_xM$. If it is, then the Finsler structure $(M, F)$ is called {\it absolutely homogeneous}, in other words, $F(x, y) = F(x, -y)$, for all $(x, y)\in TM$.

The simplest case of Finsler surface is a Riemannian surface and in this case its indicatrix is a centrally symmetric circle on an ellipse in $T_xM$, as well known.

We are going to construct a canonical moving frame on $\Sigma$ ( see for example \cite{Br1997}, \cite{Br2002}).

Let $\Sigma_1$ be the unit tangent bundle of a Riemannian metric, say $a$, on $M$ (it is customary to denote $\alpha:=\sqrt{a(y,y)}$, where $(x,y)\in TM$).  For any Finsler structure $\Sigma$ on $M$, there exists a smooth, positive function 
$r: \Sigma_1 \longrightarrow {\bf R}^+$
such that 
$$\Sigma_r = \bigg\{\frac{1}{r(u)}\cdot u : u\in\Sigma_1 \bigg\}.$$
This notations will be used throughly.

In order to assure the strong convexity on $\Sigma$, an additional differential condition on $r$ 
must be given. Conversely, any positive function $p: \Sigma_1 \longrightarrow {\bf R}^+$
satisfying a certain differential condition defines a Finsler structure on the surface $M$ in this way. In other words, one can say that a Finsler structures on a surface $M$ dependes
on a function of 3 variables, namely the function $r$ on $\Sigma_1$. Obviously $\Sigma_p$ is in fact the indicatrix bundle of $(M, F)$ and the curve $\Sigma_p\big\vert_x = 
\Big\{\frac{1}{p(u)}\cdot u : u\in\Sigma_1\big\vert_x\Big\} \equiv \{y\in T_xM : F(x, y) = 1\}$ corresponds to the indicatrix curve described above.

The function $$\rho: \Sigma_p \longrightarrow \Sigma_1, \quad \rho\bigg(\frac{1}{p(u)}\cdot u\bigg) = u, \quad \forall u\in\Sigma_1$$
is the "inverse" function which takes the Finsler structure $(M, F)$ back to the original Riemannian structure $(M, a)$.

The functions $F$ and $p$ are essentially the same, namely, if one parametrizes the Riemannian indicatrix $\Sigma_1\big\vert_x$ by the usual Euclidean angle $t$, then
$$p(x^1, x^2, t) = F(x^1, x^2, \cos t, \sin t),$$ where $\big(y^1(t), y^2(t)\big) = (\cos t, \sin t).$

Recall that a {\it Finsler space with $(\alpha, \beta)$-metric} $(M, F)$ is given by a Finslerian norm $F = F(x, y): TM \longrightarrow \lbrack 0, \infty)$, where $F$ is a
positive one-homogeneous function of the two arguments $\alpha$ and $\beta$. Hereafter we consider only $(\alpha, \beta)$-metrics obtained by means of a positively definite
Riemannian metric $(M, a)$ on $M$ and a linear 1-form $\beta(x,y)=b_i(x)y^i$, such that $a(b, b) < 1$.

Following Shen (\cite{S2}), we can always write $F$ as  
\begin{eqnarray}
F = \alpha\cdot\phi\bigg(\frac{\beta}{\alpha}\bigg),
\end{eqnarray}
where $\phi: I= \lbrack -r, r\rbrack \longrightarrow \lbrack 0, \infty)$ is a $C^{\infty}$ function and the interval $I$ can be chosen large enough such that $r \ge |\frac{\beta}{\alpha}|$, for all $x \in M$ and $y \in T_xM$. 

We also recall
\begin{lemma}$\label{le:01}
$(\cite{S2})$\,$ The function $F = \alpha\cdot\phi(s)$, $s = \frac{\beta}{\alpha}$ is a Finsler metric for any $\alpha = \sqrt{a_{ij}y^iy^j}$ and any $\beta = b_iy^i$ with $\|\beta_x\|_\alpha < b_0$ if and only if $\phi = \phi(s)$ is a positive $C^\infty$ function on $(-b_0, b_0)$ satisfying the following conditions:
\begin{equation}\label{ec:1}
\phi(s) - s\phi'(s) + (b^2 - s^2)\phi''(s) > 0, \quad \vert s\vert \leq b < b_0.
\end{equation}
\end{lemma}

\begin{remark}
\begin{enumerate}
\item Lemma \ref{le:01} implies that 
\begin{equation}\label{ec:2}
\phi(s) - s\phi'(s) > 0, \quad \vert s\vert < b_0
\end{equation}
\item In general, due to the presence of the 1-form $\beta$, the function $F$ is not absolute homogeneous.
\end{enumerate}
\end{remark}

Classical examples of $(\alpha,\beta)$- metrics are:
the {\it Randers metrics}, namely $F = \alpha + \beta$, or {\it Matsumoto metrics}, i.e. $F = \frac{\alpha^2}{\alpha - \beta}$.

For simplicity, we will use in the following the notations:
\begin{eqnarray}\label{5.16}
\phi'(s) = \frac{\partial\phi(t)}{\partial t}\bigg\vert_{t=s}, \qquad \phi'(-s) = \frac{\partial\phi(t)}{\partial t}\bigg\vert_{t=-s}.
\end{eqnarray}
In other words, we have
\begin{eqnarray}\label{5.17}
\big\lbrack\phi(-s)\big\rbrack' = \frac{d\phi(-s)}{ds} = -\phi'(-s) \\\nonumber
\big\lbrack\phi(-s)\big\rbrack'' = \frac{d^2\phi(-s)}{d^2s} = \phi''(-s).
\end{eqnarray}

Let us remark that Lemma \ref{le:01} implies 
\begin{lemma}$\label{le:02}$
Let $F = \alpha\cdot\phi\Big(\frac{\beta}{\alpha}\Big)$ be an $(\alpha,\beta)$ Finsler metric, i.e. $\phi: I \longrightarrow \lbrack 0, \infty)$ satifies condition in Lemma
\ref{le:01}. Then the reverse Finsler metric $\bar F(x, y) := F(x, -y)$ must be a Finsler metric as well.
\end{lemma}

\begin{proof}
Indeed, if $F$ is a Finsler metric, then the corresponding function $\phi$ must satisfy
\begin{eqnarray}\label{cond1}
\begin{cases}
\phi(s) - s\phi'(s) + (b^2 - s^2)\phi''(s) > 0, \quad \vert s\vert \leq b < b_0, \\
\phi(s) - s\phi'(s) > 0, \quad \vert s\vert < b_0.
\end{cases}
\end{eqnarray}
But, since $s$ belongs to an interval symmetric about $0$, the formulas above must hold good for $-s$ as well, i.e. by substituting $s$ with $-s$, we must have
\begin{eqnarray}
\begin{cases}
\phi(-s) + s\phi'(-s) + (b^2 - s^2)\phi''(-s) > 0, \quad \vert s\vert \leq b < b_0, \\
\phi(-s) + s\phi'(-s) > 0, \quad \vert s\vert < b_0.
\end{cases}
\end{eqnarray}
This conditions are, in fact, the necessary and sufficient conditions for $\bar F = \alpha\cdot\bar\phi\bigg(\frac{\beta}{\alpha}\bigg)$, $\bar \phi(s) := \phi(-s)$
to be a Finsler metric.
$\qedd$
\end{proof}

\begin{lemma}$\label{le:03}$
If $F = \alpha\cdot\phi\Big(\frac{\beta}{\alpha}\Big)$ is an $(\alpha,\beta)$ Finsler metric, then $\phi$ cannot be an odd function.
\end{lemma}
\begin{proof}
Let us assume that $F$ is Finsler and the corresponding $\phi$ is odd, i.e. $\phi(- s) = -\phi(s)$, for all $s \in (-b_0, b_0)$, then it follows
\begin{equation}
\begin{split}
& \big\lbrack\phi(-s)\big\rbrack' = \big\lbrack-\phi(s)\big\rbrack' = -\phi'(s),\\
& \big\lbrack\phi(-s)\big\rbrack'' = \big\lbrack-\phi(s)\big\rbrack'' = -\phi''(s).
\end{split}
\end{equation}

On the other hand, using the derivation rule of composed functions, we get
\begin{eqnarray*}
\big\lbrack\phi(-s)\big\rbrack' =  -\phi'(-s), \qquad \big\lbrack\phi(-s)\big\rbrack'' = \big\lbrack-\phi(s)\big\rbrack'' = \phi''(-s).
\end{eqnarray*}
and therefore, for odd functions, we obtain $\phi'(- s) = \phi'(s)$, $\phi''(- s) = -\phi''(s)$.

Substituting these formulas in \eqref{ec:1}, we have
\begin{equation*}
-\phi(-s) + s\phi'(-s) - (b^2 - s^2)\phi''(-s) > 0, \quad \vert s\vert \leq b < b_0, \\
\end{equation*}
but these formula contradicts \eqref{ec:1} written by putting $-s$ instead of $s$, i.e. it is impossible for $\phi$ to be an odd function.
$\qedd$
\end{proof}


\section{Moving frames on Finsler surfaces}

 The 3-manifold $\Sigma_1$ can be regarded as the orthonormal frame bundle over
M with respect to $a$ and therefore it has a canonical coframing $\{\alpha^1, \alpha^2, \alpha^3\}$, where $\alpha^1$, $\alpha^2$ are the tautological 1-forms and 
$\alpha^3$ is the Levi-Civita connection form. The canonical coframing $\{\alpha^1, \alpha^2, \alpha^3\}$ satisfies the structure equations 
\begin{eqnarray}\label{ecstruct00}
d\alpha^1 & = & \alpha^2\wedge\alpha^3 ,\\\nonumber
d\alpha^2 & = & \alpha^3\wedge\alpha^1,\\\nonumber
d\alpha^3 & = &  k\alpha^1\wedge\alpha^2,
\end{eqnarray}
where the function $k: M \longrightarrow {\bf R}$ is the Gauss curvature of the Riemannian structure $(M, a)$.

It is well known that for a Finsler structure $(M, F)$ with indicatrix bundle $\Sigma\subset TM$ a canonical coframing $\{\omega^1, \omega^2, \omega^3\}$ can be as well
constructed. The corresponding structure equations are
\begin{eqnarray}\label{ecstruct01}
d\omega^1 & = & -I\omega^1\wedge\omega^3 + \omega^2\wedge\omega^3 ,\\\nonumber
d\omega^2 & = & -\omega^1\wedge\omega^3,\\\nonumber
d\omega^3 & = &  K\omega^1\wedge\omega^2 - J\omega^1\wedge\omega^3,
\end{eqnarray}
where the functions $I, J, K: \Sigma \longrightarrow {\bf R}$ are called {\it the Cartan, Landsberg} and {\it flag curvatures}, respectively (see \cite{Br1997}, \cite{Br2002}, 
\cite{SSS}). We point out that, unlikely 
the Riemannian case, all these curvatures live on $\Sigma$ and not on the base manifold $M$.

Regarding now the Finslerian indicatrix bundle $\Sigma \equiv \Sigma_p$ as a deformation of the Riemannian unit tangent bundle $\Sigma_1$ by $\rho: \Sigma_p \longrightarrow \Sigma_1$, 
where $p: \Sigma_1 \longrightarrow {\bf R}^+$ gives the Finslerian norm, it is quite obvious that the cotangent map  $\rho^*: T^*\Sigma_1 \longrightarrow T^*\Sigma_p,$ will
allow to obtain the Finsler coframing $\{\omega^1, \omega^2, \omega^3\}$ from the Riemannian one $\{\alpha^1, \alpha^2, \alpha^3\}$. Indeed, some computations show 
\begin{eqnarray}\label{01}
\omega^1 & = & \rho^*\big(\sqrt{p(p + p_{33})}\alpha^1\big),\\\nonumber
\omega^2 & = & \rho^*(p\alpha^2 + p_3\alpha^1),\\\nonumber
\omega^3 & = & \rho^*\bigg(\frac{(p + p_{33})\alpha^3 + (p_{32} - p_1)\alpha^2}{\sqrt{p(p + p_{33})}} + \frac{P_p\alpha^1}{\sqrt{p^3(p + p_{33})^3}}\bigg),
\end{eqnarray}
where
\begin{eqnarray}
\nonumber
P_p & = & \frac{1}{2}(p_3p_{32}p_{33} - p_3p_{33}p_1 + pp_{333}p_{32} - pp_1p_{333} + 2pp_{32}p_3 - 2pp_1p_3\\ & - & 3pp_2p_{33} - p^2p_{332} - 2p^2p_2 - p_2p^2_{33} - pp_{332}p_{33}).
\end{eqnarray}

It can be seen that the strongly convexity of $\Sigma_r$ is equivalent to the differentiable condition, (see \cite{Br1997}, \cite{Ca})
$$p_{33} + p >0,$$
where the subscript indicate the directional derivatives with respect to the Riemannian coframing $\{\alpha^1, \alpha^2, \alpha^3\}$, i.e. for any differentiable function 
$f: \Sigma_1 \longrightarrow {\bf R}$ we denote
\begin{eqnarray}
df = f_1\cdot\alpha^1 + f_2\cdot\alpha^2 + f_3\cdot \alpha^3.
\end{eqnarray}



It is known that the geodesics of the Riemannian structure $(M, a)$ are the projections to $M$ of the integral lines of the exterior differential system $\{\alpha^1=0,  \alpha^3=0\}$
defined on $\Sigma_1$.

Similarly, for a Finsler structure $(M, F)$ with indicatrix bundle $\Sigma$ and canonical coframing $\{\omega^1, \omega^2, \omega^3\}$, the Finslerian geodesics are the 
projections to $M$ of the integral lines of the exterior differential system $\{\omega^1 = 0, \omega^3 = 0\}$ on $\Sigma$.

Let us consider now another Finsler structure $\overline F$ on the same surface $M$. This implies that there exists another smooth positive function, say 
$r:\Sigma_1 \longrightarrow {\bf R}^+$,
such that 
\begin{eqnarray}
\Sigma_r = \bigg\{\frac{1}{r(u)}\cdot u : u\in\Sigma_1 \bigg\}
\end{eqnarray}
is the indicatrix bundle of $(M, \overline F)$. The inverse function $\bar\rho: \Sigma_r \longrightarrow \Sigma_1$,
\begin{eqnarray}
\bar\rho\bigg(\frac{1}{r(u)}\cdot u\bigg) = u, \quad \forall u\in\Sigma_1
\end{eqnarray}
allows to recover the original Riemannian structure $(M, a)$.

Obviously, $\bar\rho$ is invertible in the sense that we can always define 
\begin{eqnarray}
\bar\rho^{-1}: \Sigma_1 \longrightarrow \Sigma_r, \quad \bar\rho^{-1} (u)= \frac{1}{r(u)}\cdot u, \quad \forall u\in\Sigma_1.
\end{eqnarray}

This means that the following diagram is commutative 
\begin{eqnarray*}
\begin{array}[c]{ccccc}
\Sigma_p & \xrightarrow{\mu}& \Sigma_r\\
&\rho\searrow \qquad \swarrow\bar\rho\\
& \Sigma_1 &
\end{array}
\end{eqnarray*}
where $\mu:=\bar\rho^{-1}\circ \rho$, and therefore, the 
coframings $\{\alpha^1, \alpha^2, \alpha^3\}$, $\{\omega^1, \omega^2, \omega^3\}$ and $\{\bar\omega^1, \bar\omega^2, \bar\omega^3\}$ on $\Sigma$, 
$\Sigma_p$ and $\Sigma_r$, respectively, are related as follows
\begin{eqnarray*}
\begin{array}[c]{ccccc}
\{\omega^i\}& \xleftarrow{\mu^*}& \{\overline\omega^i\}\\
&\rho^*\nwarrow \qquad \nearrow\bar\rho^*\\
& \{\alpha^i\} &
\end{array}
\end{eqnarray*}
where $\{\bar\omega^1, \bar\omega^2, \bar\omega^3\}$ is the associated canonical coframe of $(M, \overline F)$ defined in the same way as above. We also have 
\begin{eqnarray}\label{02}
\overline\omega^1 & = & \overline\rho^*\big(\sqrt{r(r + r_{33})}\alpha^1\big),\\\nonumber
\overline\omega^2 & = & \overline\rho^*(r\alpha^2 + r_3\alpha^1),\\\nonumber
\overline\omega^3 & = & \overline\rho^*\bigg(\frac{(r + r_{33})\alpha^3 + (r_{32} - r_1)\alpha^2}{\sqrt{r(r + r_{33})}} + \frac{P_r \alpha^1}{\sqrt{r^3(r + r_{33})^3}}\bigg),
\end{eqnarray}
where
\begin{eqnarray}\nonumber
P_r & = & \frac{1}{2}(r_3r_{32}r_{33} - r_3r_{33}r_1 + rr_{333}r_{32} - rr_1r_{333} + 2rr_{32}r_3 - 2rr_1r_3\\\nonumber & - & 3rr_2r_{33} - r^2r_{332} - 2r^2r_2 - r_2r^2_{33} - rr_{332}r_{33}).
\end{eqnarray}

Similar formulas can be written by means of $\mu$ in order to construct the relation between the coframings $\{\omega^1, \omega^2, \omega^3\}$ and $\{\bar\omega^1, \bar\omega^2, \bar\omega^3\}$,
but we do not need to do this.

With this setting, one can see that  the Finsler $F$ and $\overline F$ structures are projectively equivalent if and only  $\textrm{span} \{\omega^1, \omega^3 \}=\mu^*(\textrm{span} \{\bar\omega^1, \bar\omega^3 \})$.
Since both $\Sigma_p$, $\Sigma_r$ are topologically diffeomorphic to projective sphere $SM :=\widetilde{TM} / _\sim$ and $\mu : \Sigma_p \longrightarrow \Sigma_r$ is diffeomorphism,
we identify here the 3-manifolds $\Sigma_p$ and $\Sigma_r$, where the equivalence relation $\sim$ is defined by $(x, y) \sim (x, z)$ if and only if $y$, $z$ are positive 
multiples of each other.

In terms of the Riemannian canonical coframing $\{\alpha^1, \alpha^2, \alpha^3\}$ the above condition become 
$\textrm{span}\{\alpha^1,
 \mathcal M_2\cdot\alpha^2 +\mathcal M_3\cdot\alpha^3\} =\textrm{span}\{\alpha^1, \overline{\mathcal M_2}\cdot\alpha^2 +\overline{\mathcal M_3}\cdot\alpha^3\} = 0$, where we denote for simplicity
 \begin{eqnarray}
\begin{cases}\mathcal M_2 = \frac{p_{32} - p_1}{\sqrt{p(p + p_{33})}} \\
\mathcal M_3 = \sqrt{\frac{p + p_{33}}{p}} 
\end{cases},\qquad
\begin{cases}\overline{\mathcal M_2} = \frac{r_{32} - r_1}{\sqrt{r(r + r_{33})}}\\
\overline{\mathcal M_3} = \sqrt{\frac{r + r_{33}}{r}}.
\end{cases} 
\end{eqnarray}
It can be seen easily now that the projective equivalence condition reduce to 
\begin{eqnarray}
\frac{\mathcal M_3}{\mathcal M_2} = \frac{\overline{\mathcal M_3}}{\overline{\mathcal M_2}}, \Longleftrightarrow \frac{p + p_{33}}{p_{32} - p_1} = \frac{r + r_{33}}{r_{32} - r_1},
\end{eqnarray}
provided $\mathcal M_2 \not= 0$ and $\overline{\mathcal M_2} \not= 0$. We observe that the geometrical meaning of $\mathcal M_2 = 0$ is that $(M, F)$ and $(M, \alpha)$ are
projectively related, i.e. the Finslerian geodesic of $(M, F)$ and Riemannian geodesic of $(M, \alpha)$ coincide. In this case, obviously $(M, F)$, $(M, \overline F)$ and $(M, \alpha)$
are all projectively equivalent. We consider this case to be trivial and exclude it from our analysis. Therefore, we always assume in the following that the Finslerian structures 
$(M, F)$ and $(M, \overline F)$ are not projectively equivalent to $(M, \alpha)$, i.e. $\mathcal M_2 \not= 0$ and $\overline{\mathcal M_2} \not= 0$.

In order to obtain the condition for $(M, F)$ to be with reversible geodesics, we impose the condition that $\overline F(x, y) = F(x, -y)$, for all $(x, y) \in TM$, where 
$\overline F(x, y)$ is the reverse Finsler structure associated to $F$ on $M$. In this case, with the notations above, we obtain:
\begin{proposition}
Let $(M, F)$ be a Finsler surface and $(M, \overline F)$ be the associated reverse Finsler structure on $M$. We assume that both Finslerian structures $F$ and $\overline F$ are 
not Riemannian projectively equivalent. Then, $(M, F)$ is with reversible geodesics if and only if  
\begin{eqnarray}\label{cond}
\frac{p + p_{33}}{p_{32} - p_1} = \frac{r + r_{33}}{r_{32} - r_1}
\end{eqnarray}
with the notations above.
\end{proposition}


\section{The reversible geodesics condition}

We start with the Riemannian surface $(M, a)$ and let us consider the isothermal coordinates $x = (x^1, x^2)$ on $M$, namely, in these local coordinates
$a_{ij}=e^{2\nu}\delta_{ij}$, where $\nu$ is a smooth function on $M$ and $\delta_{ij}$ the Kronecker operator. This allows to write the canonical Riemannian 
coframing $\{\alpha^1, \alpha^2, \alpha^3\}$ as 
\begin{eqnarray}\label{riefin0}
\alpha^1 & = & -e^{\nu(x_1, x_2)}\sin t\;dx^1 + e^{\nu(x_1, x_2)}\cos t\;dx^2, \\\nonumber
\alpha^2 & = & e^{\nu(x_1, x_2)}\cos t\;dx^1 + e^{\nu(x_1, x_2)}\sin t\;dx^2, \\\nonumber
\alpha^3 & = & - \frac{\partial\nu(x_1, x_2)}{\partial x^2}dx^1 + \frac{\partial\nu(x_1, x_2)}{\partial x^1}dx^2 + dt,
\end{eqnarray}
where $t\in \lbrack 0, 2\pi)$ is the fiber coordinate.
The unit circle $\Sigma_1\big\vert_x\in T_xM$ of $(M, a)$ is therefore parametrized as 
\begin{eqnarray}
\begin{cases}
y^1 = e^{-\nu(x^1, x^2)}\cdot\cos t\\
y^2 = e^{-\nu(x^1, x^2)}\cdot\sin t, \qquad t\in \lbrack 0, 2\pi).
\end{cases}
\end{eqnarray}

One can easily remark that for a vector $(y^1, y^2)$, the opposite vector is given by $-y = (-y^1, -y^2) = \big(e^{-\nu(x^1, x^2)}\cdot\cos(l + \pi), e^{-\nu(x^1, x^2)}\cdot\sin(l + \pi)\big)$.
Therefore, if we denote by $p$ and $r$ the Finslerian norms corresponding to $F(x, y)$ and $\overline F(x, y) = F(x, -y)$ considered as positive real valued function on $\Sigma_1$
as explained before, then we get
\begin{lemma}
The relation between $p$ and $r$ is given by
\begin{eqnarray}\label{rel}
r(x^1, x^2, t) = p(x^1, x^2, t + \pi).
\end{eqnarray}
\end{lemma}
Straightforward computations give immediately the relations between the directional derivatives of $p$ and $r$ with respect to the Riemannian coframing 
$\{\alpha^1, \alpha^2, \alpha^3\}$ and the partial derivatives with respect to the natural coordinates $(x^1, x^2, t)$. 

We have 
\begin{eqnarray}
p_1 & = &e^{-\nu(x^1,x^2)}\bigg(-\frac{\partial p(x^1, x^2, t)}{\partial x^1}\sin t + \frac{\partial p(x^1, x^2, t)}{\partial x^2}\cos t\\\nonumber & - & \frac{\partial p(x^1, x^2, t)}{\partial t}\Big(\frac{\partial \nu(x^1, x^2)}{\partial x^1}\cos t + \frac{\partial \nu(x^1, x^2)}{\partial x^2}\sin t\Big)\bigg),\\\nonumber
p_2 & = &e^{-\nu(x^1,x^2)}\bigg(\frac{\partial p(x^1, x^2, t)}{\partial x^1}\cos t + \frac{\partial p(x^1, x^2, t)}{\partial x^2}\sin t\\\nonumber & + & \frac{\partial p(x^1, x^2, t)}{\partial t}\Big(\frac{\partial \nu(x^1, x^2)}{\partial x^2}\cos t - \frac{\partial \nu(x^1, x^2)}{\partial x^1}\sin t\Big)\bigg),\\\nonumber
p_3 & = &\frac{\partial p(x^1, x^2, t)}{\partial t},\\\nonumber
p_{31} & = & e^{-\nu(x^1,x^2)}\bigg(-\frac{\partial^2 p(x^1, x^2, t)}{\partial x^1\partial t}\sin t + \frac{\partial^2 p(x^1, x^2, t)}{\partial x^2\partial t}\cos t\\\nonumber & - & \frac{\partial^2 p(x^1, x^2, t)}{\partial t^2}\Big(\frac{\partial \nu(x^1, x^2)}{\partial x^1}\cos t + \frac{\partial \nu(x^1, x^2)}{\partial x^2}\sin t\Big)\bigg),\\\nonumber
p_{32} & = &e^{-\nu(x^1,x^2)}\bigg(\frac{\partial^2 p(x^1, x^2, t)}{\partial x^1\partial t}\cos t + \frac{\partial^2 p(x^1, x^2, t)}{\partial x^2\partial t}\sin t\\\nonumber & + & \frac{\partial^2 p(x^1, x^2, t)}{\partial t^2}\Big(\frac{\partial \nu(x^1, x^2)}{\partial x^2}\cos t - \frac{\partial \nu(x^1, x^2)}{\partial x^1}\sin t\Big)\bigg),\\\nonumber
p_{33} & = &\frac{\partial^2 p(x^1, x^2, t)}{\partial t^2}.
\end{eqnarray}
From here, it follows
\begin{eqnarray}
p + p_{33} & = & p(x_1, x_2, t) + \frac{\partial^2 p(x^1, x^2, t)}{\partial t^2},\\\nonumber
p_{32} - p_1 & = & e^{-\nu(x^1,x^2)}\bigg(\frac{\partial^2 p(x^1, x^2, t)}{\partial x^1\partial t}\cos t + \frac{\partial^2 p(x^1, x^2, t)}{\partial x^2\partial t}\sin t\\\nonumber & + & \frac{\partial^2 p(x^1, x^2, t)}{\partial t^2}\Big(\frac{\partial \nu(x^1, x^2)}{\partial x^2}\cos t - \frac{\partial \nu(x^1, x^2)}{\partial x^1}\sin t\Big)\bigg)\\\nonumber 
& - & e^{-\nu(x^1,x^2)}\bigg(-\frac{\partial p(x^1, x^2, t)}{\partial x^1}\sin t + \frac{\partial p(x^1, x^2, t)}{\partial x^2}\cos t\\\nonumber & - & \frac{\partial p(x^1, x^2, t)}{\partial t}\Big(\frac{\partial \nu(x^1, x^2)}{\partial x^1}\cos t + \frac{\partial \nu(x^1, x^2)}{\partial x^2}\sin t\Big)\bigg).
\end{eqnarray}
Similarly, taking into account of \eqref{rel}, we obtain for $r$
\begin{eqnarray}
r_1 & = &e^{-\nu(x^1,x^2)}\bigg(-\frac{\partial p(x^1, x^2, t + \pi)}{\partial x^1}\sin t + \frac{\partial p(x^1, x^2, t + \pi)}{\partial x^2}\cos t\\\nonumber & - & \frac{\partial p(x^1, x^2, t + \pi)}{\partial t}\Big(\frac{\partial \nu(x^1, x^2)}{\partial x^1}\cos t + \frac{\partial \nu(x^1, x^2)}{\partial x^2}\sin t\Big)\bigg),\\\nonumber
r_2 & = &e^{-\nu(x^1,x^2)}\bigg(\frac{\partial p(x^1, x^2, t + \pi)}{\partial x^1}\cos t + \frac{\partial p(x^1, x^2, t + \pi)}{\partial x^2}\sin t\\\nonumber & + & \frac{\partial p(x^1, x^2, t + \pi)}{\partial t}\Big(\frac{\partial \nu(x^1, x^2)}{\partial x^2}\cos t - \frac{\partial \nu(x^1, x^2)}{\partial x^1}\sin t\Big)\bigg),\\\nonumber
r_3 & = &\frac{\partial p(x^1, x^2, t + \pi)}{\partial t},\\\nonumber
r_{31} & = &e^{-\nu(x^1,x^2)}\bigg(-\frac{\partial^2 p(x^1, x^2, t + \pi)}{\partial x^1\partial t}\sin t + \frac{\partial^2 p(x^1, x^2, t + \pi)}{\partial x^2\partial t}\cos t\\\nonumber & - & \frac{\partial^2 p(x^1, x^2, t + \pi)}{\partial t^2}\Big(\frac{\partial \nu(x^1, x^2)}{\partial x^1}\cos t + \frac{\partial \nu(x^1, x^2)}{\partial x^2}\sin t\Big)\bigg),\\\nonumber
r_{32} & = &e^{-\nu(x^1,x^2)}\bigg(\frac{\partial^2 p(x^1, x^2, t + \pi)}{\partial x^1\partial t}\cos t + \frac{\partial^2 p(x^1, x^2, t + \pi)}{\partial x^2\partial t}\sin t\\\nonumber & + & \frac{\partial^2 p(x^1, x^2, t + \pi)}{\partial t^2}\Big(\frac{\partial \nu(x^1, x^2)}{\partial x^2}\cos t - \frac{\partial \nu(x^1, x^2)}{\partial x^1}\sin t\Big)\bigg),\\\nonumber
r_{33} & = &\frac{\partial^2 p(x^1, x^2, t + \pi)}{\partial t^2}
\end{eqnarray}
and
\begin{eqnarray}
r + r_{33} & = & p(x_1, x_2, \pi + t) + \frac{\partial^2 p(x^1, x^2, \pi + t)}{\partial t^2},\\\nonumber
r_{32} - r_1 & = & e^{-\nu(x^1,x^2)}\bigg(\frac{\partial^2 p(x^1, x^2, \pi + t)}{\partial x^1\partial t}\cos t + \frac{\partial^2 p(x^1, x^2, \pi + t)}{\partial x^2\partial t}\sin t\\\nonumber & + & \frac{\partial^2 p(x^1, x^2, \pi + t)}{\partial t^2}\Big(\frac{\partial \nu(x^1, x^2)}{\partial x^2}\cos t - 
\frac{\partial \nu(x^1, x^2)}{\partial x^1}\sin t\Big)\bigg)\\\nonumber & - & e^{-\nu(x^1,x^2)}\bigg(-\frac{\partial p(x^1, x^2, \pi + t)}{\partial x^1}\sin t + \frac{\partial p(x^1, x^2, \pi + t)}{\partial x^2}\cos t\\\nonumber & - & \frac{\partial p(x^1, x^2, \pi + t)}{\partial t}\Big(\frac{\partial \nu(x^1, x^2)}{\partial x^1}\cos t + \frac{\partial \nu(x^1, x^2)}{\partial x^2}\sin t\Big)\bigg).
\end{eqnarray}
Using all these formulas, the projectively equivalence condition \eqref{cond} becomes 
$$(p_{32} - p_1)\cdot(r + r_{33}) - (r_{32} - r_1)\cdot(p + p_{33}) = 0,$$ or equivalently,
\begin{eqnarray}\label{ecprinc}
& & \cos t \bigg(\Big(\frac{\partial^2p}{\partial x^1\partial t} - \frac{\partial p}{\partial x^2} + \frac{\partial p}{\partial t} \frac{\partial\nu}{\partial x^1}\Big)\Big(\frac{\partial^2r}{\partial t^2} + r\Big)\\\nonumber & - & \Big(\frac{\partial^2r}{\partial x^1\partial t} - \frac{\partial r}{\partial x^2} 
+ \frac{\partial r}{\partial t} \frac{\partial\nu}{\partial x^1}\Big)\Big(\frac{\partial^2p}{\partial t^2} + p\Big) + \frac{\partial\nu}{\partial x^2}\Big(\frac{\partial^2p}{\partial t^2}r - \frac{\partial^2r}{\partial t^2}p\Big)\bigg)\\\nonumber & + & \sin t \bigg(\Big(\frac{\partial^2p}{\partial x^2\partial t} 
+ \frac{\partial p}{\partial x^1} + \frac{\partial p}{\partial t} \frac{\partial\nu}{\partial x^2}\Big)\Big(\frac{\partial^2r}{\partial t^2} + r\Big)\\\nonumber & - & \Big(\frac{\partial^2r}{\partial x^2\partial t} + \frac{\partial r}{\partial x^1} +
 \frac{\partial r}{\partial t} \frac{\partial\nu}{\partial x^2}\Big)\Big(\frac{\partial^2p}{\partial t^2} + p\Big) - \frac{\partial\nu}{\partial x^1}\Big(\frac{\partial^2p}{\partial t^2}r - \frac{\partial^2r}{\partial t^2}p\Big)\bigg) = 0, 
\end{eqnarray}
and further we have
\begin{eqnarray}
& & \cos t \bigg(\Big(\frac{\partial^2p(x^1, x^2, t)}{\partial x^1\partial t} - \frac{\partial p(x^1, x^2, t)}{\partial x^2} +
 \frac{\partial p(x^1, x^2, t)}{\partial t} \frac{\partial\nu(x^1, x^2)}{\partial x^1}\Big)\\\nonumber && \Big(\frac{\partial^2p(x^1, x^2, \pi + t)}{\partial t^2}
 + p(x^1, x^2, \pi + t)\Big) - \Big(\frac{\partial^2p(x^1, x^2, \pi + t)}{\partial x^1\partial t}\\\nonumber & - & \frac{\partial p(x^1, x^2, \pi + t)}{\partial x^2} + 
\frac{\partial p(x^1, x^2, \pi + t)}{\partial t} \frac{\partial\nu(x^1, x^2)}{\partial x^1}\Big)\Big(\frac{\partial^2p(x^1, x^2, t)}{\partial t^2}\\\nonumber & + & p(x^1, x^2, t)\Big) + \frac{\partial\nu(x^1, x^2)}{\partial x^2}\Big(\frac{\partial^2p(x^1, x^2, t)}{\partial t^2}p(x^1, x^2, \pi + t)\\\nonumber & - & \frac{\partial^2p(x^1, x^2, \pi + t)}{\partial t^2}p(x^1, x^2, t)\Big)\bigg) 
+ \sin t \bigg(\Big(\frac{\partial^2p(x^1, x^2, t)}{\partial x^2\partial t} + \frac{\partial p(x^1, x^2, t)}{\partial x^1}\\\nonumber & + & \frac{\partial p(x^1, x^2, t)}{\partial t} \frac{\partial\nu(x^1, x^2)}{\partial x^2}\Big)\Big(\frac{\partial^2p(x^1, x^2, \pi + t)}{\partial t^2} + p(x^1, x^2, \pi + t)\Big)\\\nonumber & - & 
 \Big(\frac{\partial^2p(x^1, x^2, \pi + t)}{\partial x^2\partial t} + \frac{\partial p(x^1, x^2, \pi + t)}{\partial x^1} + \frac{\partial p(x^1, x^2, \pi + t)}{\partial t} \frac{\partial\nu(x^1, x^2)}{\partial x^2}\Big)\\\nonumber & & \Big( \frac{\partial^2p(x^1, x^2, t)}{\partial t^2} + p(x^1, x^2, t)\Big) - \frac{\partial\nu(x^1, x^2)}{\partial x^1}\Big(\frac{\partial^2p(x^1, x^2, t)}{\partial t^2}p(x^1, x^2, \pi + t)\\\nonumber & - &
 \frac{\partial^2p(x^1, x^2, \pi + t)}{\partial t^2}p(x^1, x^2, t)\Big)\bigg) = 0. 
\end{eqnarray}

Let us remark that in the natural coordinates $(x^1, x^2, t)$ on $\Sigma_1$ we have
\begin{eqnarray}
\alpha & := & \sqrt{a(y,y)} = 1, \\\nonumber
\beta & := & b_1(x^1, x^2)y^1 + b_2(x^1, x^2)y^2 \\\nonumber
&=& e^{-\nu(x^1, x^2)}\big\lbrack b_1(x^1, x^2)\cdot\cos l + b_2(x^1, x^2)\cdot\sin l\big\rbrack
\end{eqnarray}
where $\nu, b_1, b_2 : M \longrightarrow {\bf R}$ are smooth functions.

Hence, on the hypersurface $\Sigma_1\hookrightarrow TM$, we can put $s = \beta$ and therefore
\begin{eqnarray}\label{intro3}
\nonumber
p(x^1, x^2, t) & = & \phi(s)_{|s=\beta} = \phi(b_1(x^1, x^2)e^{-\nu(x_1, x_2)}\cos t\\\nonumber & + & b_2(x^1, x^2)e^{-\nu(x_1, x_2)}\sin t),\\
p(x^1, x^2, \pi +t) & = & r(x^1, x^2, t) = \phi(-s)_{|s=\beta}\\\nonumber & = & \phi\big(b_1(x^1, x^2)e^{-\nu(x_1, x_2)}\cos(\pi + t)\\\nonumber & + & b_2(x^1, x^2)e^{-\nu(x_1, x_2)}\sin(\pi + t)\big). 
\end{eqnarray}

Straightforward computations give:
\begin{eqnarray}
\frac{\partial\beta}{\partial x^1} & = & e^{-\nu(x^1, x^2)}\bigg\lbrack\Big(\frac{\partial b_1(x^1, x^2)}{\partial x^1}\cos t + \frac{\partial b_2(x^1, x^2)}{\partial x^1}\sin t\Big)\\\nonumber 
& - & \frac{\partial\nu(x^1, x^2)}{\partial x^1}\big(b_1(x^1, x^2)\cos t + b_2(x^1, x^2)\sin t\big)\bigg\rbrack,\\\nonumber
\frac{\partial\beta}{\partial x^2} & = & e^{-\nu(x^1, x^2)}\bigg\lbrack\Big(\frac{\partial b_1(x^1, x^2)}{\partial x^2}\cos t + \frac{\partial b_2(x^1, x^2)}{\partial x^2}\sin t\Big)\\\nonumber 
& - & \frac{\partial\nu(x^1, x^2)}{\partial x^2}\big(b_1(x^1, x^2)\cos t + b_2(x^1, x^2)\sin t\big)\bigg\rbrack,\\\nonumber
\frac{\partial\beta}{\partial t} & = & e^{-\nu(x^1, x^2)}\big(-b_1(x^1, x^2)\sin t + b_2(x^1, x^2)\cos t\big) = \beta'_t ,
\end{eqnarray}
\begin{eqnarray}
\frac{\partial^2\beta}{\partial x^1\partial t} & = & e^{-\nu(x^1, x^2)}\bigg\lbrack\Big(-\frac{\partial b_1(x^1, x^2)}{\partial x^1}\sin t + \frac{\partial b_2(x^1, x^2)}{\partial x^1}\cos t\Big)\\\nonumber & - & \frac{\partial\nu(x^1, x^2)}{\partial x^1}\big(-b_1(x^1, x^2)\sin t + b_2(x^1, x^2)\cos t\big)\bigg\rbrack,\\\nonumber
\frac{\partial^2\beta}{\partial x^2\partial t} & = & e^{-\nu(x^1, x^2)}\bigg\lbrack\Big(-\frac{\partial b_1(x^1, x^2)}{\partial x^2}\sin t + \frac{\partial b_2(x^1, x^2)}{\partial x^2}\cos t\Big)\\\nonumber & - & \frac{\partial\nu(x^1, x^2)}{\partial x^2}\big(-b_1(x^1, x^2)\sin t + b_2(x^1, x^2)\cos t\big)\bigg\rbrack,\\\nonumber
\frac{\partial^2\beta}{\partial t^2} & = & -\beta.
\end{eqnarray}

If we intoduce the notations
\begin{eqnarray}
\mathcal A & := & e^{-\nu(x^1, x^2)}\Big(\frac{\partial b_1(x^1, x^2)}{\partial x^1}\cos t + \frac{\partial b_2(x^1, x^2)}{\partial x^1}\sin t\Big),\\\nonumber
\mathcal B & := & e^{-\nu(x^1, x^2)}\Big(\frac{\partial b_1(x^1, x^2)}{\partial x^2}\cos t + \frac{\partial b_2(x^1, x^2)}{\partial x^2}\sin t\Big),\\\nonumber
\mathcal C & := & e^{-\nu(x^1, x^2)}\Big(-\frac{\partial b_1(x^1, x^2)}{\partial x^1}\sin t + \frac{\partial b_2(x^1, x^2)}{\partial x^1}\cos t\Big),\\\nonumber
\mathcal D & := & e^{-\nu(x^1, x^2)}\Big(-\frac{\partial b_1(x^1, x^2)}{\partial x^2}\sin t + \frac{\partial b_2(x^1, x^2)}{\partial x^2}\cos t\Big),\\\nonumber
\beta'_t & := & \frac{\partial\beta}{\partial t}
\end{eqnarray}
and
\begin{eqnarray}
 a & := & \mathcal A - \frac{\partial \nu(x^1, x^2)}{\partial x^1}\cdot\beta,\\\nonumber
 b & := & \mathcal B - \frac{\partial \nu(x^1, x^2)}{\partial x^2}\cdot\beta,\\\nonumber
 c & := & \mathcal C - \frac{\partial \nu(x^1, x^2)}{\partial x^1}\cdot\beta_t',\\\nonumber
 d & := & \mathcal D - \frac{\partial \nu(x^1, x^2)}{\partial x^2}\cdot\beta_t',\\\nonumber
\end{eqnarray}
we obtain 
\begin{eqnarray*}
\frac{\partial p(x^1, x^2, t)}{\partial x^1} & = & \phi'(s)_{|s=\beta}\Big(\mathcal A - \frac{\partial\nu(x^1, x^2)}{\partial x^1}\beta\Big) = \phi'(s)_{|s=\beta}\cdot a,\\\nonumber
\frac{\partial p(x^1, x^2, t)}{\partial x^2} & = & \phi'(s)_{|s=\beta}\Big(\mathcal B - \frac{\partial\nu(x^1, x^2)}{\partial x^2}\beta\Big) = \phi'(s)_{|s=\beta}\cdot b,\\\nonumber
\frac{\partial p(x^1, x^2, t)}{\partial t} & = & \phi'(s)_{|s=\beta}\cdot\beta'_t,\\\nonumber
\frac{\partial r(x^1, x^2, t)}{\partial x^1} & = & -\phi'(-s)_{|s=\beta}\Big(\mathcal A - \frac{\partial\nu(x^1, x^2)}{\partial x^1}\beta\Big) = -\phi'(-s)_{|s=\beta}\cdot a,\\\nonumber
\frac{\partial r(x^1, x^2, t)}{\partial x^2} & = & -\phi'(-s)_{|s=\beta}\Big(\mathcal B - \frac{\partial\nu(x^1, x^2)}{\partial x^2}\beta\Big) = \phi'(-s)_{|s=\beta}\cdot b,\\\nonumber
\frac{\partial r(x^1, x^2, t)}{\partial t} & = & -\phi'(-s)_{|s=\beta}\cdot\beta'_t\\\nonumber
\frac{\partial^2 p(x^1, x^2, t)}{\partial x^1\partial t} & = & \phi''(s)_{|s=\beta}\cdot(\beta'_t)\Big(\mathcal A - \frac{\partial\nu(x^1, x^2)}{\partial x^1}\beta\Big)\\\nonumber 
& + & \phi'(s)_{|s=\beta}\Big(\mathcal C - \frac{\partial\nu(x^1, x^2)}{\partial x^1}\beta'_t\Big)\\\nonumber & = &  \phi''(s)_{|s=\beta}\cdot(\beta'_t)\cdot a + \phi'(s)_{|s=\beta}\cdot c,\\\nonumber
\frac{\partial^2 p(x^1, x^2, t)}{\partial x^2\partial t} & = & \phi''(s)_{|s=\beta}\cdot(\beta'_t)\Big(\mathcal B - \frac{\partial\nu(x^1, x^2)}{\partial x^2}\beta\Big)\\\nonumber 
& + & \phi'(s)_{|s=\beta}\Big(\mathcal D - \frac{\partial\nu(x^1, x^2)}{\partial x^2}\beta'_t\Big)\\\nonumber & = & \phi''(s)_{|s=\beta}\cdot(\beta'_t)\cdot b + \phi'(s)_{|s=\beta}\cdot d,\\\nonumber
\frac{\partial^2 p(x^1, x^2, t)}{\partial t^2} & = & \phi''(s)_{|s=\beta}\cdot(\beta'_t)^2 - \phi'(s)_{|s=\beta}\cdot\beta,\\\nonumber
\frac{\partial^2 r(x^1, x^2, t)}{\partial x^1\partial t} & = & \phi''(-s)_{|s=\beta}\cdot(\beta'_t)\Big(\mathcal A - \frac{\partial\nu(x^1, x^2)}{\partial x^1}\beta\Big)\\\nonumber 
& - & \phi'(-s)_{|s=\beta}\Big(\mathcal C - \frac{\partial\nu(x^1, x^2)}{\partial x^1}\beta'_t\Big)\\\nonumber & = & \phi''(-s)_{|s=\beta}\cdot(\beta'_t)\cdot a - \phi'(-s)_{|s=\beta}\cdot c,\\\nonumber
\end{eqnarray*}
\begin{eqnarray*}
\frac{\partial^2 r(x^1, x^2, t)}{\partial x^2\partial t} & = & \phi''(-s)_{|s=\beta}\cdot(\beta'_t)\Big(\mathcal B - \frac{\partial\nu(x^1, x^2)}{\partial x^2}\beta\Big)\\\nonumber 
& - & \phi'(-s)_{|s=\beta}\Big(\mathcal D - \frac{\partial\nu(x^1, x^2)}{\partial x^2}\beta'_t\Big)\\\nonumber & = & \phi''(-s)_{|s=\beta}\cdot(\beta'_t)\cdot b - \phi'(-s)_{|s=\beta}\cdot d,\\\nonumber
\frac{\partial^2 r(x^1, x^2, t)}{\partial t^2} & = & \phi''(-s)_{|s=\beta}\cdot(\beta'_t)^2 + \phi'(-s)_{|s=\beta}\cdot\beta.
\end{eqnarray*}
Substituting now all these in $\eqref{ecprinc}$ and arranging convenient the terms, we obtain
\begin{eqnarray}\label{inloc}
\nonumber
& &\Big\lbrack \frac{\partial^2r}{\partial t^2} + r \Big\rbrack\cdot T_1 - \Big\lbrack \frac{\partial^2p}{\partial t^2} + p \Big\rbrack\cdot T_2 + 
\Big(\cos t\frac{\partial\nu}{\partial x^2} - \sin t\frac{\partial\nu}{\partial x^1}\Big)\cdot T_3\\ & + & \Big(\cos t\frac{\partial\nu}{\partial x^1} + 
\sin t\frac{\partial\nu}{\partial x^2}\Big)\cdot T_4 = 0,
\end{eqnarray}
where
\begin{eqnarray*}
T_1 & := & \cos t\Big(\frac{\partial^2p}{\partial x^1\partial t} - \frac{\partial p}{\partial x^2}\Big) + \sin t\Big(\frac{\partial^2p} {\partial x^2\partial t}
 + \frac{\partial p}{\partial x^1}\Big)\\\nonumber & = & \phi''(s)_{|s=\beta}\cdot\beta_t'\cdot\mathcal G + \phi'(s)_{|s=\beta} \cdot\mathcal H,\\\nonumber
T_2 & := & \cos t\Big(\frac{\partial^2r}{\partial x^1\partial t} - \frac{\partial r}{\partial x^2}\Big) + \sin t\Big(\frac{\partial^2r} {\partial x^2\partial t} + 
\frac{\partial r}{\partial x^1}\Big)\\\nonumber & = & \phi''(-s)_{|s=\beta}\cdot\beta_t'\cdot\mathcal G - \phi'(-s)_{|s=\beta} \cdot\mathcal H,\\\nonumber
T_3 & := & \frac{\partial^2p}{\partial t^2}\cdot r - \frac{\partial^2r}{\partial t^2}\cdot p \\\nonumber 
& = & \big(\beta_t'\big)^2\cdot\big\lbrack\phi''(s)_{|s=\beta}\phi(-s)_{|s=\beta} \\\nonumber
&-& \phi''(-s)_{|s=\beta}\phi(s)_{|s=\beta}\big\rbrack - \beta\cdot\big\lbrack \phi'(s)_{|s=\beta}\phi(-s)_{|s=\beta} + \phi'(-s)_{|s=\beta}\phi(s)_{|s=\beta}\big\rbrack,\\\nonumber
T_4 & := & \frac{\partial p}{\partial t}\cdot\Big\lbrack \frac{\partial^2r}{\partial t^2} + r \Big\rbrack  - 
\frac{\partial r}{\partial t}\cdot\Big\lbrack \frac{\partial^2p}{\partial t^2} + p \Big\rbrack \\\nonumber 
& = & \big(\beta_t'\big)^3\cdot\big\lbrack\phi'(s)_{|s=\beta}\phi''(-s)_{|s=\beta} \\\nonumber
&+& \phi'(-s)_{|s=\beta}\phi''(s)_{|s=\beta}\big\rbrack + \beta_t'\cdot\big\lbrack \phi'(s)_{|s=\beta}\phi(-s)_{|s=\beta} + \phi'(-s)_{|s=\beta}\phi(s)_{|s=\beta}\big\rbrack.
\end{eqnarray*}

The final forms from above were determined after some computations. Moreover, we have
\begin{eqnarray*}
T_1 & = & \cos t\big(\phi''(s)_{|s=\beta}\cdot(\beta_t')\cdot a + \phi'(s)_{|s=\beta}\cdot c - \phi'(s)_{|s=\beta}\cdot b\big) + \sin t\big(\phi''(s)_{|s=\beta}\cdot (\beta_t')\cdot b\\\nonumber & + & \phi'(s)_{|s=\beta}\cdot d + \phi'(s)_{|s=\beta}\cdot a\big) = \phi''(s)_{|s=\beta}\cdot(\beta_t')\lbrack a\cdot\cos t + b\cdot\sin t\rbrack\\\nonumber & + & \phi'(s)_{|s=\beta}\lbrack c\cdot\cos t - b\cdot\cos t + d\cdot\sin t + a\cdot\sin t\rbrack = \phi''(s)_{|s=\beta} \cdot\beta_t'\cdot\mathcal G + \phi'(s)_{|s=\beta} \cdot\mathcal H,\\\nonumber
T_2 & = & \cos t\big(\phi''(-s)_{|s=\beta}\cdot(\beta_t')\cdot a - \phi'(-s)_{|s=\beta}\cdot c + \phi'(-s)_{|s=\beta}\cdot b\big) + \sin t\big(\phi''(-s)_{|s=\beta} \cdot(\beta_t')\cdot b\\\nonumber & - & \phi'(s)_{|s=\beta}\cdot d - \phi'(s)_{|s=\beta}\cdot a\big) = \phi''(-s)_{|s=\beta}\cdot(\beta_t')\lbrack a\cdot\cos t + b\cdot\sin t\rbrack\\\nonumber & - & \phi'(-s)_{|s=\beta}\lbrack c\cdot\cos t - b\cdot\cos t + d\cdot\sin t + a\cdot\sin t\rbrack\\\nonumber & = & \phi''(-s)_{|s=\beta} \cdot\beta_t'\cdot\mathcal G - \phi'(-s)_{|s=\beta} \cdot\mathcal H,
\end{eqnarray*}
where
\begin{eqnarray*}
\mathcal G & := & e^{-\nu(x^1, x^2)}\Big(\frac{\partial b_1(x^1, x^2)}{\partial x^1}\cos^2 t + \sin t\cdot\cos t\Big(\frac{\partial b_2(x^1, x^2)}{\partial x^1} + \frac{\partial b_1(x^1, x^2)}{\partial x^2}\Big)\\\nonumber & + & \frac{\partial b_2(x^1, x^2)}{\partial x^2}\sin^2 t\Big) - \beta\Big(\frac{\partial\nu(x^1, x^2)}{\partial x^1}\cos t + \frac{\partial\nu(x^1, x^2)}{\partial x^2}\sin t\Big),\\\nonumber
\mathcal H & := & e^{-\nu(x^1, x^2)}\Big(\frac{\partial b_2(x^1, x^2)}{\partial x^1} - \frac{\partial b_1(x^1, x^2)}{\partial x^2}\Big) - \beta_t'\Big(\frac{\partial\nu(x^1, x^2)}{\partial x^1}\cos t\\\nonumber & + & \frac{\partial\nu(x^1, x^2)}{\partial x^2}\sin t\Big) + \beta\Big(\frac{\partial\nu(x^1, x^2)}{\partial x^2}\cos t - \frac{\partial\nu(x^1, x^2)}{\partial x^1}\sin t\Big).
\end{eqnarray*}

After some computations it follows
\begin{eqnarray*}
\Big\lbrack \frac{\partial^2r}{\partial t^2} + r \Big\rbrack\cdot T_1 - \Big\lbrack \frac{\partial^2p}{\partial t^2} + p \Big\rbrack\cdot T_2 = \beta_t'\cdot\mathcal G\cdot \mathcal E + \mathcal H\cdot \mathcal F.
\end{eqnarray*}

The equation \eqref{inloc} can be written now as 
\begin{eqnarray}\label{inloc 1}
\beta_t'\mathcal G\mathcal E + \mathcal H\mathcal F + \Big(\cos t\frac{\partial\nu}{\partial x^2} - \sin t\frac{\partial\nu}{\partial x^1}\Big)T_3 + \Big(\cos t\frac{\partial\nu}{\partial x^1} + \sin t\frac{\partial\nu}{\partial x^2}\Big)T_4 = 0.
\end{eqnarray}
If we denote
\begin{eqnarray}
\nu_+ := \cos t\cdot\frac{\partial\nu}{\partial x_1} + \sin t\cdot\frac{\partial\nu}{\partial x_2},\\\nonumber
\nu_- := \cos t\cdot\frac{\partial\nu}{\partial x_2} - \sin t\cdot\frac{\partial\nu}{\partial x_1},\\\nonumber
\end{eqnarray}
then, the formulas for $\mathcal G$ and $\mathcal H$ can be also be written as
\begin{eqnarray}
\mathcal G & = & a\cdot\cos t + b\cdot\sin t = \big(\mathcal A\cos t +\mathcal B\sin t\big) - \beta\cdot\nu_+,\\\nonumber
\mathcal H & = & (c - b)\cdot\cos t + (a + d)\cdot\sin t\\\nonumber & = & -\beta'_t\cdot\nu_+ + \beta\cdot\nu_- + \big\lbrack(b - \mathcal B)\cdot\cos t + (\mathcal A + \mathcal D)\cdot\sin t\big\rbrack.
\end{eqnarray}
By substituting these formulas in \eqref{inloc 1}, let us observe that one can group the terms and factorize by $\nu_+$ and $\nu_-$, such that we obtain
\begin{eqnarray*}
\beta'_t\cdot\mathcal E\big(\mathcal A\cdot\cos t & + & \mathcal B\sin t\big) + coeff_+\cdot\nu_+ + \mathcal F\lbrack(b - \mathcal B)\cdot\cos t + (\mathcal A + \mathcal D)\sin t\rbrack
\\\nonumber & + & coeff_-\cdot\nu_- = 0.
\end{eqnarray*}
A straightforward computation shows that 
\begin{eqnarray*}
coeff_+ & := & -\beta'_t\cdot\beta\cdot\mathcal E -\beta'_t\cdot F + T_4 = -\beta'_t\cdot\beta\cdot\mathcal E,\\\nonumber
coeff_- & := & T_3 +\mathcal F\cdot\beta = \big(\beta'_t\big)^2\cdot\mathcal E.
\end{eqnarray*}

\begin{remark}
It is useful to see that for $\beta'_t = e^{-\nu(x^1, x^2)}\big(-b_1(x^1, x^2)\sin t + b_2(x^1, x^2)\cos t\big)$ we have
\begin{eqnarray}
\nonumber
(\beta'_t)^2 & = & e^{-2\nu(x^1, x^2)}\big(b_1^2(x^1, x^2)\sin^2 t -2b_1(x^1, x^2)b_2(x^1, x^2)\sin t\cos t \\\nonumber & + & b_2^2(x^1, x^2)\cos^2 t\big)
= e^{-2\nu(x^1, x^2)}\big(b_1^2(x^1, x^2)(1 - \cos^2 t)\\ & - & 2b_1(x^1, x^2)b_2(x^1, x^2)\sin t\cos t + b_2^2(x^1, x^2)(1 - \sin^2 t)\big) \\\nonumber
& = & e^{-2\nu(x^1, x^2)}\big((b_1^2 + b_2^2) - (b_1\cos t + b_2\sin t)^2\big) = b^2 -\beta^2,
\end{eqnarray}
where, $b^2 = e^{-2\nu(x^1, x^2)}(b_1^2 + b_2^2)$ is the Riemannian length of the vector $(b_1,b_2)$.
\end{remark}

Using these formulas, we have

\begin{theorem}\label{thm:5.1}
The necessary and sufficient condition for the Finsler structures $F(x, y)$ and $\overline F(x, y) = F(x - y)$ to be projectively equivalent is 
\begin{eqnarray}\label{ecrez}
\sqrt{b^2 -s^2}\cdot \mathcal E(s)\cdot\mathcal M + \mathcal F(s)\cdot e^{-\nu(x^1, x^2)}curl_{21}=0,
\end{eqnarray}
where 
\begin{equation}\label{e}
\mathcal E(s) := s\big(\phi'(s)\phi''(-s) + \phi'(-s)\phi''(s)\big) + \big(\phi(-s)\phi'' (s) - \phi(s)\phi''(-s)\big),
\end{equation}
\begin{eqnarray}\label{ef}
\mathcal F(s) := (b^2-s^2)\big(\phi'(s)\phi''(-s) + \phi'(-s)\phi''(s)\big) + \big(\phi(-s)\phi'(s) + \phi(s)\phi'(-s)\big)
\end{eqnarray}
and
\begin{eqnarray}\label{em}
\nonumber
\mathcal M & := & e^{-\nu(x^1, x^2)}\cdot\Big(\frac{\partial b_1(x^1, x^2)}{\partial x^1}\cos^2 t + \sin t\cos t\Big(\frac{\partial b_1(x^1, x^2)}{\partial x^2} + \frac{\partial b_2(x^1, x^2)}{\partial x^1}\Big)\\\nonumber & + & \frac{\partial b_2(x^1, x^2)}{\partial x^2}\sin^2t\Big) + \beta'_t\Big(\frac{\partial\nu(x^1, x^2)}{\partial x^2}\cos t - \frac{\partial\nu(x^1, x^2)}{\partial x^1}\sin t\Big)\\ & - & \beta\Big(\frac{\partial\nu(x^1, x^2)}{\partial x^1}\cos t + \frac{\partial\nu(x^1, x^2)}{\partial x^2}\sin t\Big),\\\nonumber
curl_{21} & := & \frac{\partial b_2(x^1, x^2)}{\partial x^1} - \frac{\partial b_1(x^1, x^2)}{\partial x^2}.
\end{eqnarray}
\end{theorem}

\begin{remark}
Using the above formulas for $\beta$ and $\beta'_t$, one can see that $\mathcal M$ can be expressed as 
\begin{eqnarray}\label{em2}
\mathcal M = \mathcal K_1 + \mathcal K_2\cdot\cos2t + \mathcal K_3\cdot\sin2t,
\end{eqnarray}
where 
\begin{eqnarray*}
\mathcal K_1 & := &\frac{1}{2}\Big(\frac{\partial b_1}{\partial x^1} + \frac{\partial b_2}{\partial x^2}\Big), \\\nonumber
\mathcal K_2 & := & \frac{1}{2}\Big(\frac{\partial b_1}{\partial x^1} - \frac{\partial b_2}{\partial x^2}\Big) - \Big(\frac{\partial\nu}{\partial x^1}b_1 - \frac{\partial \nu}{\partial x^2}b_2\Big), \\\nonumber
\mathcal K_3 & := & \frac{1}{2}\Big(\frac{\partial b_2}{\partial x^1} + \frac{\partial b_1}{\partial x^2}\Big) - \Big(\frac{\partial\nu}{\partial x^2}b_1 + \frac{\partial \nu}{\partial x^1}b_2\Big). \\\nonumber
\end{eqnarray*}
\end{remark}


\section{Basic Lemmas}

In the present section, we are going to give 
some results to be used later.

\begin{lemma} \label{lem:6.1}
The following relations are equivalent 
\begin{enumerate}
\item
$\mathcal E = 0$,
\item
$\phi(s) = k_1\cdot\phi(-s) + k_2\cdot s$, $k_1$, $k_2$ non vanishing constants,
\item $F(\alpha, \beta) = F_0(\alpha, \beta) + \varepsilon\beta,
$
where $F_0$ is an absolute homogeneous $(\alpha, \beta)$-metric and $\varepsilon$ is a non vanishing constant.
\end{enumerate}
\end{lemma}
\begin{proof}
The equivalence of 1 and 2 follows directly from lemma 3.4 in \cite{MSS}. Indeed, one can easily see that $\mathcal E = 0$ is equivalent to the equation $\mathcal T - \overline{\mathcal T} = 0$ in \cite{MSS}
and therefore 2 follows.

We prove now the equivalence of 2 and 3. 
First of all, we remark that $k_1$ can take only the value 1. Indeed, by putting $-s$ instead of $s$ in relation 2, it follows
\begin{eqnarray}
\phi(-s) = k_1\cdot\phi(s) - k_2\cdot s
\end{eqnarray}
and by adding these formulas, it results
\begin{eqnarray}
\phi(s) + \phi(-s) = k_1\cdot\big\lbrack\phi(s) + \phi(-s)\big\rbrack,
\end{eqnarray}
i.e.
\begin{eqnarray}
\big\lbrack\phi(s) + \phi(-s)\big\rbrack(k_1 - 1) = 0
\end{eqnarray}
and we have two cases here. The first case is $\phi(s) = -\phi(-s)$, i.e. $\phi$ is an odd function, but this is  not good due to Lemma \ref{le:03}. Therefore, the only 
possible case is $k_1 = 1$ and the formula in 2 actually reads
\begin{eqnarray}\label{ecv2}
\phi(s) = \phi(-s) + k_2\cdot s,
\end{eqnarray}
where $k_2\not= 0$, because otherwise we would obtain only absolute homogeneous metrics. We will show now that \eqref{ecv2} is, in fact, equivalent to the relation 3.

Let us recall that the vector space of all real-valued functions is the direct sum of the subspaces of even and odd functions. In other words, any function $\phi(s)$ can be 
uniquely written as the sum of an even function $\phi_{even}$ and an odd function $\phi_{odd}$, namely
\begin{eqnarray}
\phi(s) = \phi_{even}(s) + \phi_{odd}(s),
\end{eqnarray}
where 
\begin{eqnarray}
\phi_{even}(s) = \frac{1}{2}\big\lbrack\phi(s) + \phi(-s)\big\rbrack,\\\nonumber
\phi_{odd}(s) = \frac{1}{2}\big\lbrack\phi(s) - \phi(-s)\big\rbrack.
\end{eqnarray}
Using now \eqref{ecv2} it follows
\begin{eqnarray}
\phi_{odd}(s) = \frac{1}{2}\big\lbrack\phi(s) - \phi(-s)\big\rbrack = \frac{k_2}{2}\cdot s
\end{eqnarray}
and therefore
\begin{eqnarray}
\phi(s) = \phi_{even}(s) + \frac{k_2}{2}\cdot s,
\end{eqnarray}
i.e. the corresponding $F(\alpha, \beta)$ is of the form 
 in 3.
$\qedd$

\end{proof}

We are going to discuss next the equation $\mathcal F (s)= 0$, where  $\mathcal F (s)$ is given in 
\eqref{ef}.

A straightforward computation shows that, for $\phi'(s) \not= 0$, this is equivalent to
\begin{eqnarray}\label{ef4}
\nonumber
\frac{(b^2 - s^2)\cdot\bar\phi''(s) - s\bar\phi'(s) + \bar\phi(s)}{\bar\phi'(s)} = \frac{(b^2 - s^2)\cdot\phi''(s) - s\phi'(s) + \phi(s)}{\phi'(s)},\\
\end{eqnarray}
where we put $\bar\phi(s) := \phi(-s)$. Since both $\phi$ and $\bar\phi$ must be Finsler metrics, from Lemma \ref{le:01} it results that the numerators in both hand sides of 
\eqref{ef4} must be positive and from here it results $\phi'(s)\cdot\bar\phi'(s) > 0$, in other words, $\phi$ and $\bar\phi$ must have the same monotonicity.

Let us remark that every even function $\phi$ is solution of $\mathcal F = 0$. Of course, any odd function is also solution, but we can exclude these functions due to Lemma \ref{le:03}.

Let us suppose that an arbitrary $\phi$, i.e. it is not even, nor odd, is solution of $\mathcal F = 0$. Then, $\phi(s)$ and $\phi(-s)$ must have the same monotonicity. We will
show that this is not possible.

Indeed, recall that the composition of two functions with same monotonicity gives an increasing function and the composition of two functions with different monotony gives
 an decreasing function (this can be easily be seen from the derivation rule of composed functions).

If we  write $\bar\phi(s) = (\phi\circ \psi)(s)$, where $\psi(s) := -s$, then we have two cases
\begin{enumerate}
\item 
If $\phi$ is an increasing function, then, since, $\psi$ is decreasing, their composition $\bar\phi(s)$ is decreasing, i.e. $\phi(s)$ and $\bar\phi(s)$ have different
monotonicities, but this is contradiction.
\item
If $\phi$ is decreasing, it follows that $\bar\phi(s)$ is increasing, but this also implies that $\phi(s)$ and $\bar\phi(s)$ have different monotonicities and this is 
not good again.
\end{enumerate}
We can conclude that the equation $\mathcal F = 0$ has no Finslerian solution, except the absolute homogeneous Finsler metrics, provided $\phi'(s) \not= 0$ 
for all $s \in (-b_0, b_0).$

Let us consider now the case $\phi'(s) = 0$. 

If $\phi'(s) = 0$ for all $s\in\lbrack - b_0, b_0 \rbrack$, then $\phi$ is linear in $s$ and this is not good because we do not get a 
genuine Finsler metric.

Therefore, the only possible case is that there exists some $s_0\in\lbrack - b_0, b_0 \rbrack$ such that $\phi'(s_0) = 0$, i.e. $s_0$ is a singular point of $\phi$. 
In order to study the metric at the singular point $s_0$, we need to consider the $2^{nd}$ order derivative $\phi''(s_0)$.

Let us assume that $s_0$ is degenerate, i.e. $\phi''(s_0) = 0$. Then, by Taylor's expansion, $\phi$ must be of the form $\phi(s) = a + c\cdot s^3 +$ {\it higher order
terms}, for $s_0 - \varepsilon < s < s_0 + \varepsilon$. Consequently, by neglecting the higher order terms, we obtain $\phi'(s) = 3cs^2$ and $\bar \phi'(s) = -3cs^2$, for
$\varepsilon \rightarrow 0$. Thus we get $\phi'(s)\cdot\bar \phi'(s) < 0$, but this is a contradiction.

Therefore, all singular points $s_0$ must be non-degenerate, i.e. $\phi: \lbrack - b_0, b_0 \rbrack \longrightarrow {\bf R^+}$ is a Morse function. From Morse theory, we know
that the set of singular points of $\phi$ must be finite and contains only isolated points. Then, by means of Morse lemma, it follows that $\phi$ must be of the form $\phi(s) = 
a + bs^2 +$ {\it higher orders terms}, for $s_0 - \varepsilon < s < s_0 + \varepsilon$.

One can easily verify that, for example, 
\begin{equation*}
\phi(s)  =  a + B\cdot s^2,\qquad
\phi(s)  =  c\cdot s^4,
\end{equation*}  
etc, are solutions of $\mathcal F = 0$.\\

In general, one can see that for arbitrary $s$, the function 
\begin{eqnarray*}
\phi(s) = a_0 + a_2\cdot s^2 + a_4\cdot s^4 + ... + a_{2k}\cdot s^{2k}
\end{eqnarray*}  
is a solution for $\mathcal F = 0$, but this is an even function, i.e. $F$ must be absolute homogeneous, and from the previous analysis it follows that there are no other solutions
of the equation $\mathcal F = 0$.

\begin{remark}
We point out that for a singular point $s_0$ of $\phi$, there exists a small enough positive constant $\varepsilon$ such that there is no other singular point in the
$\varepsilon$-neighborhood $(s_0 - \varepsilon, s_0 + \varepsilon)$. Indeed, if the singular points would accumulate, then $\phi$ must be constant on the $\varepsilon$-neighborhood
and it is not good for us because violates the conditions in Lemma \ref{le:01}.
\end{remark}

Therefore, we can conclude
\begin{proposition}\label{prop:5.2}
The equation $\mathcal F = 0$ has no other Finsler solutions except the absolute homogeneous case.
\end{proposition}

We also have
\begin{lemma}\label{E odd}
The function $\mathcal E(s)$ is an odd function and $\mathcal F(s)$ is an even one. 
\end{lemma}
\begin{proof}
Indeed, if one puts $-s$ instead of $s$ in the definitions of $\mathcal E(s)$ and $\mathcal F(s)$, then the conclusion follows immediately. Here, we take into account the formulas (\ref{5.16}) and (\ref{5.17}).
$\qedd$
\end{proof}

\section{$(\alpha,\beta)$ -metrics with reversible geodesics} 

Let us consider the necessary and sufficient condition (\ref{ecrez}) given in Theorem 
\ref{thm:5.1} for an $(\alpha,\beta)$-metric to be with reversible 
geodesics. 

If we put $-s$ instead of $s$ and taking into account Lemma \ref{E odd} it follows
\begin{equation}\label{-ecrez}
\begin{split}
& \sqrt{b^2 -s^2}\cdot \mathcal E(-s)\cdot\mathcal M + \mathcal F(-s)\cdot e^{-\nu(x^1, x^2)}curl_{21}=0\\
& -\sqrt{b^2 -s^2}\cdot \mathcal E(s)\cdot\mathcal M + \mathcal F(s)\cdot e^{-\nu(x^1, x^2)}curl_{21}=0,
\end{split}
\end{equation}
and therefore, from relations \eqref{ecrez} and \eqref{-ecrez} it follows
\begin{equation}\label{E-F system}
\begin{cases}
& \mathcal E(s)\cdot\mathcal M=0\\
& \mathcal F(s)\cdot curl_{21}=0.
\end{cases}
\end{equation}

Since, due to Proposition \ref{prop:5.2}, the condition  $\mathcal F(s)=0$ is not convenient, it follows that geodesic reversibility condition \eqref{ecrez} is equivalent to one of the following two cases
\begin{equation}
\mathcal E(s)=0,\quad curl_{21}=0,
\end{equation}
or
\begin{equation}
\mathcal M=0,\quad curl_{21}=0.
\end{equation}

The first case was already discussed in Lemma \ref{lem:6.1}. 

We will discuss next the case $\mathcal M = 0$.

We start assuming $\mathcal M = 0$, for all $t \in \lbrack0, 2\pi)$, i.e. $\mathcal M = \mathcal K_1 + \mathcal K_2\cdot\cos2t + \mathcal K_3\cdot\sin2t = 0$.
Evaluating this formula in $t = 0$, $t = \frac{\pi}{2}$ and $\frac{\pi}{4}$, we get $\mathcal K_1 = \mathcal K_2 = \mathcal K_3=0$, and taking into account the condition $curl_{21}=0$, we obtain 
\begin{eqnarray}\label{sistem}
\begin{cases}
\frac{\partial b_2}{\partial x^1} - \frac{\partial b_1}{\partial x^2} = 0, \\
\frac{\partial b_1}{\partial x^1} + \frac{\partial b_2}{\partial x^2} = 0, \\
\frac{1}{2}\Big(\frac{\partial b_1}{\partial x^1} - \frac{\partial b_2}{\partial x^2}\Big) - \Big(\frac{\partial\nu}{\partial x^1}b_1 - \frac{\partial \nu}{\partial x^2}b_2\Big) = 0, \\
\frac{1}{2}\Big(\frac{\partial b_2}{\partial x^1} + \frac{\partial b_1}{\partial x^2}\Big) - \Big(\frac{\partial\nu}{\partial x^2}b_1 + \frac{\partial \nu}{\partial x^1}b_2\Big)=0. 
\end{cases}
\end{eqnarray}
This is a $1^{st}$ order $PDE$ with 2 unknown functions $b_1$, $b_2$, defined on $M$, where $\nu$ is a given function.

One can easily remark that the first two equations of the system are in fact the divergence and the curl of the vector $(b_1,b_2)$ and these are equivalent to Riemann-Cauchy conditions for differentiability of the function $\mathfrak b:\mathbb C\to \mathbb C$, given by $\mathfrak b(x^1,x^2)=(b_1(x^1,x^2), b_2(x^1,x^2))$. In other words, any differentiable complex function of one complex variable on $M$ satisfies the first two equations of the system (\ref{sistem}). 

By writing these two relations as
\begin{equation}
\frac{\partial b_2}{\partial x^1} = \frac{\partial b_1}{\partial x^2},  \qquad 
\frac{\partial b_2}{\partial x^2} =-\frac{\partial b_1}{\partial x^1} , 
\end{equation}
the remaining two equations read
\begin{equation}
\begin{split}
&\frac{\partial b_1}{\partial x^1} = \frac{\partial\nu}{\partial x^1}b_1 - \frac{\partial \nu}{\partial x^2}b_2 \\
&\frac{\partial b_2}{\partial x^1} = \frac{\partial\nu}{\partial x^1}b_2 + \frac{\partial \nu}{\partial x^2}b_1. 
\end{split}
\end{equation}

A straight forward computation shows that this PDE system is integrable, i.e. 
$$\frac{\partial}{\partial x^2}\Bigl(\frac{\partial b_1}{\partial x^1}\Bigl)-
\frac{\partial}{\partial x^1}\Bigl(\frac{\partial b_1}{\partial x^2}\Bigl)=0,$$
 if and only if 
\begin{equation}\label{laplacian}
\frac{\partial^2\nu}{\partial x^1\partial x^1} +\frac{\partial^2\nu}{\partial x^2\partial x^2} =0,
\end{equation}
provided $b_1$ and $b_2$ do not vanish in the same time. 

We remark that the same conclusion follows from the Cartan-K\"ahler theory applied to the system 
(\ref{sistem}).

On the other hand, we recall that in the isothermal coordinates $x^1$, $x^2$, the Gauss curvature $k$ of the Riemannian metric $e^{2\nu}\delta_{ij}$ is given by
\begin{equation}
k=-e^{-2\nu}\Bigl(\frac{\partial^2\nu}{\partial x^1\partial x^1} +\frac{\partial^2\nu}{\partial x^2\partial x^2} \Bigr).
\end{equation}

Therefore we can conclude that the PDE system (\ref{sistem}) is integrable if and only if the Riemannian metric $a$ is flat. But this means that the function $\nu(x^1,x^2)$ must be constant and thus the system 
(\ref{sistem}) has only the constant solution, i.e. the functions $b_1$, $b_2$ are constant. \\

\begin{remark}
If $(M, a)$ is a flat Riemannian space and $\beta = b_i\cdot y^i$ a linear 1-form on $TM$, such that $b_1$, $b_2$ are constants, then any $(\alpha, \beta)$ metric $F = F(\alpha, \beta)$
constructed with these $\alpha$ and $\beta$ is with reversible geodesics and projectively equivalent $(M, a)$. In fact, $F$ is a Minkowski metric on $M$.

Indeed, one can easily see that if $(M, a)$ is flat and $b_i$ constants, then this implies that the geodesic spray  coefficients of $F(\alpha, \beta)$ are simply (see \cite{MSS}, Prop. 2.1.)
\begin{eqnarray*}
G^i(x, y) = \frac{1}{2}\gamma^i_{00} = \frac{1}{2}\gamma^i_{jk}(x)\cdot y^iy^j\equiv 0,
\end{eqnarray*}
i.e. Finslerian geodesics coincide with Riemannian ones which are straight lines in plane.

We point out that this property is true in arbitrary dimension.
\end{remark}

From the previous discussion, it follows that our analysis lead us to the following two classes of 2-dimensional $(\alpha, \beta)$- Finsler metrics with reversible geodesics

\begin{center}
\begin{tabular}{|c|c|c|c|c|}
\hline
& & \\
  Class                  & $ F$   &    $\alpha, \beta$\\
                    & & \\
\hline
& & \\
$A$  & $F(\alpha, \beta) =  F_0(\alpha, \beta) + \varepsilon\cdot\beta$   &      $\beta$: closed 1-form, $\alpha$: arbitrary\\
& & \\
\hline
& & \\
$B$  &  $   F(\alpha, \beta)$: arbitrary  &     $ b_1, b_2$: constants, $\alpha$: Euclidean flat \\
& & \\
\hline
\end{tabular}
\end{center}
{\bf Table 1.} Classes of Finsler surfaces with $(\alpha,\beta)$-metrics that have reversible geodesics. Here $F_0$ is an absolute homogeneous Finsler metric, and $\varepsilon$ a non vanishing constant. 
\\

Therefore, we may state our main result 
\begin{theorem}
A 2-dimensional Finsler space with $(\alpha, \beta)$-metrics is with reversible geodesics if and only if it belongs to one of the classes described in the table above.
\end{theorem}


\bigskip

\medskip

\begin{center}
Ioana M. Masca 

\bigskip

'Nicolae Titulescu' College Brasov,\\
B-dul 13 Decembrie Nr. 125, \\
Brasov, Romania

\bigskip
Sorin V. Sabau and H. Shimada

\bigskip

Department of Mathematics\\
Tokai  University\\
Sapporo City, Hokkaido\\ 
005\,--\,8601 Japan

\bigskip

{\small
$\bullet$\,our e-mail addresses\,$\bullet$

\medskip
\textit{e-mail of Masca}\,:

\medskip
{\tt ioana.masca@imaf.ro}

\bigskip 
\textit{e-mail of Sabau} \,:

\medskip
{\tt sorin@tspirit.tokai-u.jp}

\bigskip 
\textit{e-mail of Shimada} \,:

\medskip
{\tt shimadah@tokai-u.jp}

}
\end{center}


\begin{thebibliography}{MMMM}
\bibitem [BM]{BM}
 S. {Bacso} and M. {Matsumoto},
 {\it Projective changes between Finsler spaces with $(\alpha,\beta)$-metric}, Tensor, N.S., {\bf 55} (1994), 252--257.

\bibitem[Br1]{Br1997}
	       {Bryant},~R.,
{\it Projectively flat Finsler 2-spheres of constant curvature}, Selecta
		 Math. (N.S.), {\bf vol. 3, no. 2} (1997), 161--203.

\bibitem[Br2]{Br2002}
	       {Bryant},~R.,
{\it Some remarks on Finsler manifolds with constant flag curvature},
Houston Journal of Mathematics, {\bf vol. 28, no.2} (2002), 221--262.  

\bibitem[Br3]{Br}
     R. {Bryant}, 
{\it Geodesically reversible Finsler 2-spheres on constant curvature}, in {\it Inspired by S.S. Chern}, (Ed. Phillip A. Griffiths), Nankai Tracts in Math. Vol.11(2006), 95 -- 111.   

\bibitem[Ca]{Ca}
      C. {Catone},
{\it Projective equivalence of Finsler and Riemannian surfaces}, Diff. Geom. Appl., {\bf 26 (4)} (2008), 404--418. 

\bibitem[Cr]{Cr}
          M. {Crampin}, 
{\it Randers spaces with reversible geodesics}, Publ. Math. Debrecen, {\bf 67/3-4}(2005), 401--409. 

\bibitem[MSS]{MSS}
      I.M. {Masca}, S.V. {Sabau}, H. {Shimada}
{\it Reversible geodesics for $(\alpha, \beta)$ metrics}, Intl. Jour. Math., {\bf 21 (8)} (2010), 1071 -- 1094. 

\bibitem[SSS]{SSS}
      S.V. {Sabau}, K. {Shibuya}, H. {Shimada}
{\it On the existence of generalized unicorns on surfaces}, Diff. Geom. Appl., {\bf 28} (2010), 406--435.

\bibitem[S1]{S1}
      Z. {Shen},
{\it Differential Geometry of Sprays and Finsler Spaces}, Kluwer Academic Publishers, 2001.

\bibitem[S2]{S2}
      Z. {Shen},
{\it Landsberg Curvature, $S$-Curvature and Riemann Curvature}, in {\it A Sampler of Riemann-Finsler Geometry}, (Ed. Bao, D., Bryant, R.L., Chern, S.S., Shen, Z.,), MSRI Publications, Cambridge University Press, Vol.{\bf 50}(2004), 303--355. 


\end{thebibliography}
\end{document}